\documentclass[11pt]{amsart}
\usepackage[utf8]{inputenc}
\usepackage[T1]{fontenc}
\usepackage{lmodern}
\usepackage{amsmath}
\usepackage{amssymb}
\usepackage{amsfonts}
\usepackage{mathrsfs}
\usepackage{amsthm}
\usepackage{commath}
\usepackage{ stmaryrd }

\newcommand{\R}{\mathbb{R}}

\newcommand\restr[2]{{
		\left.\kern-\nulldelimiterspace 
		#1 
		\right|_{#2} 
	}}

\renewcommand{\H}{\mathscr{H}}

\newcommand{\Mp}{\mathscr{M}_p}
\newcommand{\diam}{\mathrm{diam}}

\newcommand{\Ep}{\mathcal{E}_p}
\newcommand{\Epq}{\mathcal{E}_{p,q}}
\newcommand{\K}{\mathcal{K}}

\newtheorem{theorem}{Theorem}[section]
\newtheorem*{theorem*}{Theorem}
\newtheorem{lemma}[theorem]{Lemma}
\newtheorem{cor}[theorem]{Corollary}
\theoremstyle{definition}
\newtheorem{remark}[theorem]{Remark}
\newtheorem{defi}[theorem]{Definition}

\title[Characterization of Sobolev-Slobodeckij spaces...]{Characterization of Sobolev-Slobodeckij spaces using curvature energies}

\author{Damian D\k{a}browski}
\address{Damian D\k{a}browski\newline Departament de Matem\`atiques, Universitat Aut\`onoma de Barcelona; Barcelona Graduate School of Mathematics (BGSMath)\newline Edifici C Facultat de Ci\`encies, 08193 Bellaterra (Barcelona), Spain}
\email{ddabrowski@mat.uab.cat}
\date{}
\keywords{Sobolev-Slobodeckij spaces, geometric curvature energies, Menger curvature}
\subjclass[2010]{53A07, 46E35}

\begin{document}
	\begin{abstract}
		We give a new characterization of Sobolev-Slobodeckij spaces $W^{1+s,p}$ for $n/p<1+s$, where $n$ is the dimension of the domain. To achieve this we introduce a family of curvature energies inspired by the classical concept of integral Menger curvature. We prove that a function belongs to a Sobolev-Slobodeckij space if and only if it is in $L^p$ and the appropriate energy is finite.
	\end{abstract}
	
	\maketitle
	\section{Introduction}
	The aim of this paper is to give a new characterization of Sobolev-Slobodeckij spaces $W^{1+s,p}$ for $n/p<1+s$, where $n$ is the dimension of the domain. To achieve this we introduce a family of curvature energies inspired by the classical concept of integral Menger curvature. Their purpose is to measure the regularity of surfaces and how much do they ``bend''. We prove that a function belongs to a Sobolev-Slobodeckij space if and only if it is in $L^p$ and the appropriate energy is finite.
	
	\subsection*{Integral Menger curvature}
	Given three distinct points $x,y,z\in\R^n$ we denote by $R(x,y,z)$ their circumradius, i.e. the radius of the unique circle passing through them (for $x,y,z$ collinear we assume $R(x,y,z)=\infty$). The inverse of $R(x,y,z)$ will be called \emph{Menger curvature} of $x,y,z$ and denoted by $c(x,y,z)$.
	
	Motivated by the search for particularly regular, \emph{optimal} shapes of knots, Gonzalez and Maddocks proposed in \cite{gonzalez1999global} to study the following functionals on the space of curves
	\begin{align*}
	\mathscr{U}_p(\gamma)&=\int_{\gamma}\ \sup_{y,z\in\gamma}c(x,y,z)^p\ d\H^1(x),\\
	\Mp(\gamma) &= \int_{\gamma}\int_{\gamma}\int_{\gamma}\ c(x,y,z)^p\ d\H^1(x)\ d\H^1(y)\ d\H^1(z).
	\end{align*}
	Later on Strzelecki, Szumańska and von der Mosel introduced in \cite{strzelecki2009geometric} an intermediate functional
	\[
	\mathscr{I}_p(\gamma) = \int_{\gamma}\int_{\gamma}\ \sup_{z\in\gamma}\ c(x,y,z)^p\ d\H^1(x)\ d\H^1(y).
	\]
	The functional $\Mp$ is called \emph{integral Menger curvature}. 
	
	The idea behind those functionals (also called \emph{knot energies} or just energies) was the following: $c(x,y,z)$ is big for $x,y,z$ close to each other, unless they happen to be almost collinear (note that for sufficiently smooth $\gamma$ the quantity $c(x,y,z)$ converges as $y,z\rightarrow x$ to the classical curvature of $\gamma$ at $x$). Therefore, the functionals  should penalize self-intersections, lack of smoothness and ``bending''. By minimizing an energy inside some fixed knot class we should find an optimal shape of this knot. Strzelecki, Szumańska and von der Mosel have shown in \cite{strzelecki2007rectifiable,strzelecki2009geometric,strzelecki2010regularizing} that for suitable values of $p$ all listed energies exhibit certain regularizing and self-repulsive properties. In \cite{strzelecki2013some} they proved results important from a knot-theoretic point of view, for example existence of minimizers inside knot classes.
	
	Interestingly, before Gonzalez and Maddocks proposed to investigate $\Mp$ in context of knot theory, a similar concept had arisen in harmonic analysis. Melnikov introduced in \cite{mel1995analytic} Menger curvature of a positive Borel measure	$\mu$ in $\mathbb{C}$ as
	\begin{equation*}
	c^2(\mu)=\iiint c(z,\zeta,w)^2\ d\mu(z)d\mu(\zeta)d\mu(w).
	\end{equation*}
	The notion has proven very useful for studying the Cauchy transform and analytic capacity, it was one of the key tools used to prove the Vitushkin's conjecture. For more information see the books \cite{pajot2002analytic,tolsa2014analytic}.
	
	\subsection*{Higher dimensional analogues}
	
	Several different attempts at generalizing integral Menger curvature to higher dimensional objects have been made. The obvious idea of integrating the inverse of the radius of an $n$-dimensional sphere passing through $n+2$ points doesn't seem to work well because there are examples of smooth and embedded surfaces for which such quantity is unbounded, see \cite[Appendix B]{strzelecki2011integral}. Several better generalizations were introduced and studied in \cite{lerman2011high,lerman2009high,strzelecki2011integral,kolasinski2012thesis,kolasinski2015geometric,kolasinski2015compactness}. 
	
	We will concentrate on the following one due to Kolasiński \cite{kolasinski2012thesis}:
	for $x_0,\dots,x_{n+1}\in\R^{n+m}$ we define
	\[
	\mathcal{K}(x_0,\dots,x_{n+1}) = \frac{\H^{n+1}(\Delta(x_0,\dots,x_{n+1}))}{\diam(x_0,\dots,x_{n+1})^{n+2}},
	\]
	where $\Delta(x_0,\dots,x_{n+1})$ stands for the convex hull of $x_0,\dots,x_{n+1}$. It is motivated by one of the formulas used to calculate Menger curvature
	\[c(x,y,z)=\frac{1}{R(x,y,z)}=4\frac{\H^2(\Delta(x,y,z))}{|x-y||y-z||z-x|}.
	\]
	Given an $n$-dimensional surface $\Sigma$ we define its \emph{integral Menger curvature} as
	\[
	\Ep(\Sigma) = \int_{\Sigma^{n+2}} \mathcal{K}(x_0,\dots,x_{n+1})^p\ d\H^{n(n+2)}(x_0,\dots,x_{n+1}). 
	\]
	
	Note that for $n=1$ we get a slightly different energy than $\Mp$. Even though it is clear that
	\[4\K(x,y,z)\leq c(x,y,z),
	\]
	in general the two quantities are not comparable: think of triples of points $x,y,z$ lying on $S^1$ such that $x$ and $y$ are fixed, but $z\rightarrow y$. $c(x,y,z)$ is constantly equal to 1, while $\K(x,y,z)$ converges to zero.

	\subsection*{The connection between Sobolev-Slobodeckij spaces and curvature energies}
	The first to notice a connection between Sobolev spaces and curvature energies of Menger type were Strzelecki and von der Mosel who proved in \cite{strzelecki2007rectifiable} that for $p>1$ and a closed curve $\gamma$ we have $\mathscr{U}_p(\gamma)<\infty$ if and only if $\gamma$ is embedded and its arclength parametrization belongs to the Sobolev space $W^{2,p}$.
	
	Blatt achieved a similar characterization of finite energy curves for $\mathscr{I}_p$ and $\Mp$ in \cite{blatt2013note}. He showed that for $p>2$ and a closed curve $\gamma$ with arclength parametrization $\Gamma$ locally a homeomorphism, $\mathscr{I}_p(\gamma)<\infty$ if and only if $\gamma$ is embedded and $\Gamma\in W^{2-1/p,p}$. Similarly, for $p>3$  and a closed curve $\gamma$ with arclength parametrization $\Gamma$ locally a homeomorphism, $\mathscr{M}_p(\gamma)<\infty$ if and only if $\gamma$ is embedded and $\Gamma\in W^{2-2/p,p}$.
	
	In \cite{blatt2012sharp} Blatt and Kolasiński described surfaces with finite $\Ep$ energy.
	\begin{theorem*}[{\cite[Theorem 1.1]{blatt2012sharp}}]
		Let $m,n\in\mathbb{N},\ p\in\R$ satisfy $n(n+1)<p<\infty$. Furthermore, let $\Sigma\subset\R^{n+m}$ be a compact $n$-dimensional $C^1$ manifold and $s=1-\frac{n(n+1)}{p}\in (0,1)$. Then $\Ep(\Sigma)$ is finite if and only if $\Sigma$ can be locally represented as the graph of a function belonging to the Sobolev-Slobodeckij space $W^{1+s,p}(\R^n, \R^{m}).$
	\end{theorem*}

	All results above used Sobolev-Slobodeckij spaces as a tool to characterize objects with finite curvature energies. The aim of this paper is to do the opposite: we use appropriately defined curvature energies to characterize spaces $W^{1+s,p}$ for as many values of $s$ and $p$ as possible.
	
	Let $U\subset\R^n$ be open, $f:U\rightarrow\R$ be measurable. Throughout the article we will use the notation 
	\begin{equation*}
	F(x)=(x,f(x))\in\R^{n+1}.
	\end{equation*}
	We define a family of energies
	\[
	\mathcal{E}_{p,q}(f) = \int_{U^{n+2}} \K_{p,q}(x_0,\dots,x_{n+1})\ dx_0 \dots dx_{n+1},
	\]
	where
	\[
	\K_{p,q}(x_0,\dots,x_{n+1})=\frac{\H^{n+1}\big(\Delta(F(x_0),\dots,F(x_{n+1})\big)^p}{\diam(x_0,\dots,x_{n+1})^{(n+2)q}}.
	\]
	Note that for $f$ Lipschitz continuous the quantity $\mathcal{E}_{p,p}(f)$ is comparable to $\Ep(\text{graph} (f))$. 
	
	We adapt the ideas from \cite{blatt2012sharp} to $\mathcal{E}_{p,q}$ and obtain the following.
	\begin{theorem}\label{thm:mythm}
		Let $n\in\mathbb{N},\ 0<s<1,\ 1<p<\infty$ satisfy $n/p<1+s$. Suppose that $U \subset\R^n$ is open, bounded and satisfies the cone condition from Definition \ref{def:cone_condition}, or $U=\R^n$. Let $q= \frac{n(n+1)}{n+2} + \frac{p(n+1+s)}{n+2}$. Then $f\in W^{1+s,p}(U)$ if and only if $f\in L^p(U)$ and $\mathcal{E}_{p,q}(f)<\infty$. Furthermore, there exists a constant $C=C(n,p,s, U)$ such that
		\begin{equation*}
		C^{-1}\lVert f\rVert^p_{W^{1+s,p}(U)}\le \lVert f\rVert_{L^p( U)}^p + \Epq(f)\le C\lVert f\rVert^p_{W^{1+s,p}(U)}.
		\end{equation*}
	\end{theorem}
	In fact, for $U=\R^n$ we prove something more.
	\begin{theorem}\label{thm:mythm in Rn}
		Let $n\in\mathbb{N},\ 0<s<1,\ 1<p<\infty$ satisfy $n/p<1+s$, and let $q= \frac{n(n+1)}{n+2} + \frac{p(n+1+s)}{n+2}$.  For all $f:\R^n\rightarrow\R$ measurable we have $\Epq(f)<\infty$ if and only if the seminorm $[f]_{W^{1+s,p}(\R^n)}$ is finite. Furthermore, there exists a constant $C=C(n,p,s)$ such that
		\begin{equation*}
		C^{-1}[f]^p_{W^{1+s,p}(\R^n)}\le \Epq(f)\le C[f]^p_{W^{1+s,p}(\R^n)}.
		\end{equation*}
	\end{theorem}
	
	Organization of the paper is the following. In Section 2 we recall some facts about Sobolev-Slobodeckij spaces. In Section 3 we prove that for $f\in W^{1+s,p}$ we have $\Epq(f)<\infty$. In Section 4 we prove the reverse implication, and thus we conclude the proof of Theorem \ref{thm:mythm} and \thmref{thm:mythm in Rn}. Our reasoning is essentially a modified version of the one in \cite{blatt2012sharp}.
	
	Throughout the article $\textbf{B}(x,r)$ will denote a closed ball of radius $r$ centered at $x$. $\omega_k$ is a constant equal to the Lebesgue measure of a $k$-dimensional unit ball. We will use the letter $C$ to denote a constant which may change from line to line and which may depend on several parameters. Any such dependence will be noted.
\section{Sobolev-Slobodeckij spaces}	
	
	Let us recall the definition of Sobolev-Slobodeckij spaces.
	\begin{defi}
		Let $U\subset\R^n$ be an open set, $k\in\{0,1,2,\dots\}, 0<s<1, 1\leq p< \infty$. Set
		\begin{equation*}
		\lVert f\rVert_{W^{k+s,p}(U)} = \lVert f\rVert_{W^{k,p}(U)} + \left(\sum_{|\alpha|=k}\int_{U}\int_{U}\frac{|D^{\alpha}f(x)-D^{\alpha}f(y)|^p}{|x-y|^{n+sp}}\ dx\ dy \right)^{1/p}.
		\end{equation*}
		Here we assume that $W^{0,p} = L^p.$ The \emph{Sobolev-Slobodeckij spaces} are defined as
		\begin{equation*}
		W^{k+s,p}(U) = \left\{f\in W^{k,p}(U)\ :\ \lVert f\rVert_{W^{k+s,p}(U)}<\infty\right\}. 
		\end{equation*}
	\end{defi}
	
	We will be working with open bounded sets satisfying the following cone condition.
	
	\begin{defi}\label{def:cone_condition}
		We say that an open bounded set $U\subset\R^n$ satisfies the \emph{cone condition}, if there exist bounded open sets $U_1,\dots,U_m$ and cones $C_1,\dots,C_m$ which are rotated versions of a fixed cone $K_h=\{(x',x_n)\in\R^n:0<x_n<h,|x'|<ax_n\}$, such that
		\[\partial U\subset\bigcup_{i=1}^{m}U_i\qquad\mathrm{and}\qquad (U\cap U_i)+C_i\subset U\]
		for each $i=1,\dots,m$.
	\end{defi}
	
	An example of sets satisfying the cone condition are open bounded sets with Lipschitz boundary.
	
	In our later considerations we will need the following well-known results about Sobolev-Slobodeckij spaces.
	
	\begin{theorem}[{\cite[4.2.3/Theorem]{triebel1978interpolation}}]\label{thm:extension}
		Let $U\subset\R^n$ be a bounded open set satisfying the cone condition, $k\in\{0,1,2,\dots\}, 0<s<1, 1<p<\infty$. Then there exists a bounded extension operator from $W^{k+s,p}(U)$ to $W^{k+s,p}(\R^n)$.
	\end{theorem}	
	
	Given a fixed open set $U\subset\R^n$ and $x\in U$ we define
	\[H_x=\{h\in\R^n: x+h\in U,x-h\in U\}
	\]
	and
	\begin{equation*}
	[f]_{W^{1+s,p}( U)} =  \left(\int_{ U}\int_{H_x}\frac{|f(x+h)-2f(x)+f(x-h)|^p}{|h|^{n+(1+s)p}}\ dh\ dx \right)^{1/p}.
	\end{equation*}

	\begin{theorem}[{\cite[2.5.1/Theorem, 4.4.2/Theorem 2]{triebel1978interpolation}}]\label{thm:equivalent_norm}
		Let $ U\subset\R^n$ be a bounded open set satisfying the cone condition, or $ U=\R^n$. Suppose $0<s<1,\ 1< p< \infty$. Then $\lVert\cdot\rVert_{L^p( U)}+[\cdot]_{W^{1+s,p}( U)}$ is an equivalent norm on $W^{1+s,p}( U)$.
		
	\end{theorem}	
	
	
	\begin{remark}
		Referenced results are stated in \cite{triebel1978interpolation} for Besov spaces $B^s_{p,q}$, but for sufficiently regular open sets (e.g. $\R^n$ or bounded open sets which satisfy the cone condition) we have $W^{s,p}( U)=B^s_{p,p}( U)$, see chapters 2.5.1, 4.4.2 in {\cite{triebel1978interpolation}}.
	\end{remark}	
	
	In Section 3 we will use the following characterization of functions with $[f]_{W^{1+s,p}(\R^n)}<\infty$ due to Dorronsoro \cite{dorronsoro1985mean}.
	Given a locally integrable function $f$ and a cube $Q\subset\R^n$ we denote by $P_Q f$ the unique affine function such that
	\begin{align*}
	&\int_Q f-P_Qf\ dx =0,\\
	&\int_Q (f-P_Qf)x_i\ dx =0,\qquad i=1,\dots,n.
	\end{align*}
	For $x\in\R^n,\ t>0$ we set
	\begin{equation*}
	\Omega_f(x,t)=\sup_Q \lVert f-P_Q\rVert_{L^{\infty}(Q)},
	\end{equation*}
	where the supremum is taken over all cubes $Q\subset\R^n$ of sidelength $t$ such that $x\in Q$.
	
	\begin{theorem}[{\cite[Theorem 2]{dorronsoro1985mean}}]\label{thm:dorronsoro}
		Let $1+s>n/p$. For any measurable function $f:\R^n\rightarrow\R$ we have $[f]_{W^{1+s,p}(\R^n)}<\infty$ if and only if
		\begin{equation*}
		\llbracket f\rrbracket_{W^{1+s,p}(\R^n)}:=\left(\int_0^{\infty}\int_{\R^n}\frac{\Omega_f(x,t)^p}{t^{1+p(1+s)}}\ dx\ dt\right)^{1/p}<\infty.
		\end{equation*}
		Moreover, we have some absolute constant $C$ such that
		\begin{equation*} 
		C^{-1}[f]_{W^{1+s,p}(\R^n)}\le \llbracket f\rrbracket_{W^{1+s,p}(\R^n)}\le C [f]_{W^{1+s,p}(\R^n)}.
		\end{equation*}
	\end{theorem}
	\begin{remark}
		The seminorm used in \cite{dorronsoro1985mean} is different from $[\cdot]_{W^{1+s,p}(\R^n)}$. However, both seminorms are equivalent, see \cite[5.2.3/Theorem 2]{triebel}.
	\end{remark}

\section{Estimating $\Epq(f)$ in terms of $[f]_{W^{1+s,p}}$}
We begin by considering the case $U=\R^n$.
	\begin{lemma}\label{thm:sobolev_implies_energy}
		Let $n\in\mathbb{N},\ 0<s<1,\ 1<p<\infty$ satisfy $\ n/p<1+s$. Let $q=\frac{n(n+1)}{n+2} + \frac{p(n+1+s)}{n+2}$. Suppose that $f:\R^n\rightarrow\R$ is measurable and $[f]_{W^{1+s,p}(\R^n)}<\infty$. Then 
		\begin{equation*} 
		\mathcal{E}_{p,q}(f)\le C [f]_{W^{1+s,p}(\R^n)}^p,
		\end{equation*}
		where $C=C(n,p,s)$.
	\end{lemma}
	
%
	The following lemma will let us use $\Omega_f(x,t)$ to estimate $\Epq$. Recall that $F(x)=(x,f(x))$.
	\begin{lemma}\label{lem:beta_est}
		Suppose $f\in L^1_{loc}(\R^n),\ x_0,\dots,x_{n+1}\in\R^n$. Then
		\begin{equation*}\label{eq:beta_est}
		\H^{n+1}(\Delta(F(x_0),\dots,F(x_{n+1})))\leq C\, \Omega_f(x_0,2\,\diam(x_0,\dots,x_{n+1}))\diam(x_0,\dots,x_{n+1})^n,
		\end{equation*}
		where $C=C(n)$.
	\end{lemma}
	
	\begin{proof}
		
		Set $d=\diam(x_0,x_1,\dots,x_{n+1}),\ T= \Delta(F(x_0),\dots,F(x_{n+1}))$.
		Let $Q\subset \R^n$ be the cube centered at $x_0$ with sidelength $2d$. Note that $x_i\in Q,\ i=0,\dots,n+1$.
		
		Without loss of generality we may assume that $P_Q(0)=0$, i.e. it is linear. We define $\Pi:\R^{n+1}\rightarrow\text{graph} (P_Q)$ as the orthogonal projection onto $\text{graph} (P_Q)$, and $\Pi^{\perp}:\R^{n+1}\rightarrow\text{graph} (P_Q)^{\perp}$ as the orthogonal projection onto $\text{graph}(P_Q)^{\perp}$.

		For every $y\in Q$ holds
		\begin{equation*}
		|f(y)-P_Q(y)|\leq \lVert f-P_Q\rVert_{L^{\infty}(Q)}\le \Omega_f(x_0,2d).
		\end{equation*}
		In particular, we have for the vertices of $T$
		\begin{equation*}
		|\Pi^{\perp}(F(x_i))|\le|F(x_i)-(x_i,P_Q(x_i))|=|f(x_i)-P_Q(x_i)|\leq \Omega_f(x_0,2d).
		\end{equation*} 
		This fact together with convexity of $T$ imply that for all $t\in T$
		\begin{equation*}\label{eq:ineq_on_T}
		|\Pi^{\perp}(t)|\leq \Omega_f(x_0,2d).
		\end{equation*}
		At the same time
		\begin{equation*}
		|\Pi(t)-\Pi(x_0)|\leq|t-x_0|\leq d.
		\end{equation*}
		Thus, $T$ is contained in
		\begin{equation*}
		Z := \{y\in\R^{n+1}\ :\ |\Pi(y)-\Pi(x_0)|\leq d,\ |\Pi^{\perp}(y)|\leq \Omega_f(x_0,2d)\}.
		\end{equation*}
		Using Fubini's theorem yields the desired inequality:
		\begin{equation*}\label{eq:tplussinequality}
		\H^{n+1}(T)\leq \H^{n+1}(Z)=2d^n\omega_n\Omega_f(x_0,2d).
		\end{equation*}
	\end{proof}
	
	\begin{proof}[Proof of \lemref{thm:sobolev_implies_energy}]
		For $x,y\in\R^n$ set
		\[U(x,y) = \big\lbrace(x_1,\dots,x_n)\in(\R^n)^n\ :\ \diam(x,x_1,\dots,x_n,y) = |x-y| \big\rbrace.
		\]
		Then due to the symmetricity of $\mathcal{K}_{p,q}(x_0,\dots,x_{n+1})$ with respect to the permutations of variables we obtain
		\begin{multline*}
		\Epq(f) = \int_{U^{n+2}}\mathcal{K}_{p,q}(x_0,\dots,x_{n+1})\ dx_0\dots dx_{n+1}\\
		= {n+2\choose 2}\int_{\R^n}\int_{\R^n}\int_{U(x,y)}\mathcal{K}_{p,q}(x,x_1\dots,x_{n},y)\ d\H^{n^2}(x_1,\dots,x_{n})\ dx\ dy.
		\end{multline*}		
		
		Recall that $F(x)=(x,f(x))$. We proceed by using definitions of $\mathcal{K}_{p,q}$ and $U(x,y)$:
		\begin{multline*}
		\Epq(f) \\= C \int_{\R^n}\int_{\R^n}\int_{U(x,y)}\frac{\H^{n+1}(\Delta(F(x),F(x_1),\dots,F(x_{n}),F(y)))^p}{\diam(x,x_1,\dots,x_n,y)^{q(n+2)}}\ d\H^{n^2}(x_1,\dots,x_{n})\ dx\ dy\\
		=C \int_{\R^n}\int_{\R^n}\int_{U(x,y)}\frac{\H^{n+1}(\Delta(F(x),F(x_1),\dots,F(x_{n}),F(y)))^p}{|x-y|^{q(n+2)}}\ d\H^{n^2}(x_1,\dots,x_{n})\ dx\ dy\\
		\overset{\mathrm{Lemma}\ \ref{lem:beta_est}.}{\leq} C \int_{\R^n}\int_{\R^n}\int_{U(x,y)} \frac{\Omega_f(x,2|x-y|)^p}{|x-y|^{q(n+2)-pn}}\ d\H^{n^2}(x_1,\dots,x_{n})\ dx\ dy\\
		= C\int_{\R^n}\int_{\R^n}\H^{n^2}(U(x,y))\frac{\Omega_f(x,2|x-y|)^p}{|x-y|^{q(n+2)-pn}}\ dx\ dy.
		\end{multline*}
		Since $U(x,y)\subset(\mathbf{B}^{n}(x,|x-y|))^n$ we have
		\[
		\H^{n^2}(U(x,y))\leq C|x-y|^{n^2}.
		\]
		Thus
		\begin{equation*}
		\Epq(f)	\leq C\int_{\R^n}\int_{\R^n}|x-y|^{n^2}\frac{\Omega_f(x,2|x-y|)^p}{|x-y|^{q(n+2)-pn}}\ dx\ dy
		= C\int_{\R^n}\int_{\R^n}\frac{\Omega_f(x,2|x-y|)^p}{|x-y|^{n+(1+s)p}}\ dx\ dy.
		\end{equation*}
		In the last equality we used the fact that $q = \frac{n(n+1)}{n+2} + \frac{p(n+1+s)}{n+2}$. 
		We change the $y$ variable using $n$-dimensional spherical coordinates such that $2|y-x|=r$, and we get from \thmref{thm:dorronsoro}
		\begin{equation*}
		\Epq(f)\le C\int_0^{\infty}\int_{\R^n}\frac{\Omega_f(x,r)^p}{r^{1+(1+s)p}}\ dx\ dr= C \llbracket f\rrbracket^p_{W^{1+s,p}(\R^n)}\le C[f]^p_{W^{1+s,p}(\R^n)}.
		\end{equation*}
	\end{proof}
	\begin{cor}
		Let $n\in\mathbb{N},\ 0<s<1,\ 1<p<\infty$ satisfy $n/p<1+s$. Suppose that $U \subset\R^n$ is open, bounded and satisfies the cone condition, or $U=\R^n$. Let $q= \frac{n(n+1)}{n+2} + \frac{p(n+1+s)}{n+2}$. If $f\in W^{1+s,p}(U),$ then
		\begin{equation*}
		\lVert f\rVert_{L^p( U)}^p + \Epq(f)\le C\lVert f\rVert^p_{W^{1+s,p}(U)}.
		\end{equation*}
	\end{cor}
	\begin{proof}
		
		For $U=\R^n$ we just use \lemref{thm:sobolev_implies_energy}, so assume $U\not=\R^n$. We use \thmref{thm:extension} to extend $f$ to $\tilde{f}\in W^{1+s,p}(\R^n)$. By \lemref{thm:sobolev_implies_energy} we get 
		\begin{equation*}
		\Epq(\tilde{f})\le C\lVert \tilde{f}\rVert_{W^{1+s,p}(\R^n)}\le C\lVert f\rVert_{W^{1+s,p}(U)}.
		\end{equation*}
		Since $\Epq(f)\le\Epq(\tilde{f})$, and $\lVert f\rVert_{L^p(U)}\le \lVert f\rVert_{W^{1+s,p}(U)}$, we are done.
		\end{proof}

\section{Estimating $[f]_{W^{1+s,p}}$ in terms of $\Epq(f)$}
	
	\begin{lemma}\label{lem:energy_implies_sobolev}
		Let $n\in\mathbb{N},\ 1<p<\infty,\ 0<s<1$. Suppose that $ U\subset\R^n$ is open, bounded and satisfies the cone condition, or $U=\R^n$. Let $q=\frac{n(n+1)}{n+2} + \frac{p(n+1+s)}{n+2}$. If $f: U\rightarrow\R$ is measurable, and $\mathcal{E}_{p,q}(f)<\infty$, then 
		\begin{equation*}
		[f]^p_{W^{1+s,p}( U)}\le C\Epq(f),
		\end{equation*}
		where $C=C(n,p,s,U).$
	\end{lemma}
	
	In the proof it will be convenient to use exterior product. The definition suitable for our purposes is given below.
	
	\begin{defi}
		Let $k\in\{1,\dots,n\}$. Given vectors $w_1,\dots,w_k\in\R^n$ we define their exterior product $w_1\wedge\dots\wedge w_k$ as a vector in $\R^{n\choose k}$ with coordinates equal to $k$-minors of the $k\times n$-matrix $(w_1,\dots,w_k)$. The coordinates are indexed by $k$-tuples $(i_1,\dots,i_k)$ with $i_j\in\{1,\dots,n\}$ and $i_1<\dots<i_k$.
	\end{defi}
	
	\begin{remark}
		We will use only two properties of exterior product, namely that
		\begin{itemize}
			\item[a)] the mapping $(w_1,\dots, w_k)\mapsto w_1\wedge\dots\wedge w_k$ is $k$-linear,
			\item[b)] the length $|w_1\wedge\dots\wedge w_k|$ is equal to the $k$-dimensional volume of the parallelotope spanned by $w_1,\dots,w_k$.
		\end{itemize}
	\end{remark}
	To prove \lemref{lem:energy_implies_sobolev} we will also need the following technical lemma.
	
	\begin{lemma}\label{lem:alpha_tilde}
		Let $ U\subset\R^n$ be an open bounded set satisfying the cone condition, or $U=\R^n$. For a fixed $\alpha\in (0,1), r\in(0,\diam( U))$ and $x\in U$ let us set
		\[
		W_{r,\alpha}^x = \left\{(w_1,w_2,\dots,w_{n})\in(\textbf{B}(0,r))^n\subset (\R^n)^n : x+w_i\in U, |w_1\wedge\dots\wedge w_n|\geq \alpha r^n\right\}.
		\]
		Then there exists $\tilde{\alpha}\in(0,1)$ such that for all $r\in(0,\diam( U))$
		\[\H^{n^2}(W^x_{r,\tilde{\alpha}}) \ge Cr^{n^2},
		\]
		where $C=C(n,U)$.		
	\end{lemma}
	
	\begin{proof}
		Let	
		\[	W_{r,\alpha} = \left\{(w_1,w_2,\dots,w_{n})\in(\textbf{B}(0,r))^n : |w_1\wedge\dots\wedge w_n|\geq\alpha r^n\right\}.
		\]
		Note that
		\[	W_{r,\alpha} = rW_{1,\alpha},
		\]
		hence 
		\begin{equation}\label{eq:wr_scaling}
		\H^{n^2}(W_{r,\alpha}) = \H^{n^2}(W_{1,\alpha})r^{n^2}.
		\end{equation}
		For $U=\R^n$ we take $\tilde{\alpha}=1/2$ and we are done because $W^x_{r,1/2}=W_{r,1/2}$ and $\H^{n^2}(W_{1,1/2})>0$. Now suppose $U$ is open, bounded, and satisfies the cone condition.
		
		Note that the set
		\[
		N = \left\{(w_1,w_2,\dots,w_{n})\in(\textbf{B}(0,1))^n : |w_1\wedge\dots\wedge w_n|=0\right\}
		\]
		is contained in the set of singular $n\times n$ matrices, so it is of zero $\H^{n^2}$ measure.
		
		Therefore, since $\bigcup_{\alpha\in(0,1)}W_{1,\alpha}\cup N= (\textbf{B}(0,1))^n$ we get
		\begin{equation}\label{eq:lim_alpha}
		\lim_{\alpha\rightarrow 0}\H^{n^2}(W_{1,\alpha}) = (\omega_n)^n.	
		\end{equation}

		Observe that due to the cone condition satisfied by $ U$ there exists a $C_0\in(0,1)$ such that for any $y\in U$ and any $r\in (0,\diam( U))$ we have
		\begin{equation}\label{eq:ball_estimate}
		\H^n( U\cap \textbf{B}(y,r))\geq C_{0}\omega_nr^n.
		\end{equation} 
		By (\ref{eq:lim_alpha}) we may choose $\tilde{\alpha}$ so small that
		\begin{equation}\label{eq:alphatilde}
		\H^{n^2}(W_{1,\tilde{\alpha}})>\left(1-\frac{C_{0}^n}{2}\right)(\omega_n)^n.
		\end{equation} 
		Without loss of generality assume that $x=0$. Then 
		\begin{equation*}
		W^x_{r,\tilde{\alpha}} = W_{r,\tilde{\alpha}}\cap U^n=W_{r,\tilde{\alpha}}\cap( U\cap\textbf{B}(0,r))^n.
		\end{equation*} It follows from (\ref{eq:wr_scaling}) and (\ref{eq:alphatilde}) that 
		\begin{equation}\label{eq:wr_estimate}
		\H^{n^2}(W_{r,\tilde{\alpha}})>\left(1-\frac{C_{0}^n}{2}\right)(\omega_n)^nr^{n^2}.	
		\end{equation}
		At the same time (\ref{eq:ball_estimate}) yields
		\begin{equation}\label{eq:omegacap_est}
		\H^{n^2}(( U\cap \textbf{B}(0,r))^n)\geq C_{0}^n(\omega_n)^nr^{n^2}.
		\end{equation}
		Since both $W_{r,\tilde{\alpha}}$ and $( U\cap \textbf{B}(0,r))^n$ are subsets of $(\textbf{B}(0,r))^n$ we have
		\begin{equation}\label{eq:est_subset_ball}
		\H^{n^2}(W_{r,\tilde{\alpha}}\cup( U\cap\textbf{B}(0,r))^n)\leq (\omega_n)^nr^{n^2}.
		\end{equation}
		The trivial equality
		\begin{multline*}
		\H^{n^2}(W_{r,\tilde{\alpha}}\cap( U\cap\textbf{B}(0,r))^n) = \H^{n^2}(W_{r,\tilde{\alpha}})+ \H^{n^2}(( U\cap \textbf{B}(0,r))^n) \\
		-\H^{n^2}(W_{r,\tilde{\alpha}}\cup( U\cap\textbf{B}(0,r))^n)
		\end{multline*}
		together with (\ref{eq:wr_estimate}), (\ref{eq:omegacap_est}) and (\ref{eq:est_subset_ball}) give us
		\begin{equation*}
		\H^{n^2}(W_{r,\tilde{\alpha}}\cap( U\cap\textbf{B}(0,r))^n) > \left(1-\frac{C_{0}^n}{2}\right)(\omega_n)^nr^{n^2} +  C_{0}^n(\omega_n)^nr^{n^2} - (\omega_n)^nr^{n^2} = \frac{C_{0}^n}{2}(\omega_n)^nr^{n^2}.
		\end{equation*}
		Thus $\H^{n^2}(W_{r,\tilde{\alpha}}\cap( U\cap\textbf{B}(0,r))^n)=\H^{n^2}(W^x_{r,\tilde{\alpha}}) \geq Cr^{n^2}$.
	\end{proof}

	\begin{proof}[Proof of \lemref{lem:energy_implies_sobolev}]

%
		Recall that $F(x)=(x,f(x))$. We have
		\begin{multline*}
		\Epq(f) = \int_{ U^{n+2}}\frac{\H^{n+1}(\Delta(F(x_0),\dots,F(x_{n+1})))^p}{\diam(x_0,\dots,x_{n+1})^{(n+2)q}}\ dx_0\dots dx_{n+1}\\
		= C\int_{ U^{n+2}}\frac{|(F(x_1)-F(x_0))\wedge\dots\wedge (F(x_{n+1})-F(x_{0}))|^p}{\diam(x_0,\dots,x_{n+1})^{(n+2)q}}\ dx_0\dots dx_{n+1}.
		\end{multline*}	
		Recall that
		\begin{equation*}
		H_x=\{h\in\R^n : x-h\in U, x+h\in U\}.
		\end{equation*}
		For $x\in U,h\in H_x$ set
		\begin{equation*}
		W^x_h:=W^x_{|h|,\tilde{\alpha}}=\left\{(w_1,w_2,\dots,w_{n})\in(\textbf{B}(0,|h|))^n : x+w_i\in U, |w_1\wedge\dots\wedge w_n|\geq \tilde{\alpha} |h|^n\right\},
		\end{equation*}
		with $\tilde{\alpha}$ given by Lemma \ref{lem:alpha_tilde} (the assumptions of Lemma \ref{lem:alpha_tilde} are met because $H_x\subset\textbf{B}(0,\diam( U)),$ so $|h|\in[0,\diam( U))$ ).
		Using the change of variables $x_0=x, x_1=x+h,x_i=x+w_{i-1}$ for $i=2,\dots,n+1$, and restricting the area of integration to $h\in H_x, (w_1,\dots,w_n)\in W_h^x$ yields 
		\begin{multline*}
		\Epq(f)= C\int_{ U^{n+2}}\frac{|(F(x_1)-F(x_0))\wedge\dots\wedge (F(x_{n+1})-F(x_{0}))|^p}{\diam(x_0,\dots,x_{n+1})^{(n+2)q}}\ dx_0\dots dx_{n+1}
		\\\geq C\int_{ U}\int_{H_x}\int_{W^x_h}|(F(x+h)-F(x))\wedge(F(x+w_1)-F(x))\wedge\dots\wedge (F(x+w_n)-F(x))|^p\\
		\cdot\frac{1}{\diam(x, x+h,x+w_1\dots,x+w_n)^{(n+2)q}}\ dw_1 \dots dw_n\ dh\ dx\\
		\geq C\int_{ U}\int_{H_x}\int_{W^x_h}|(F(x+h)-F(x))\wedge(F(x+w_1)-F(x))\wedge\dots\wedge (F(x+w_n)-F(x))|^p\\
		\cdot |h|^{-(n+2)q}\ dw_1 \dots dw_n\ dh\ dx,
		\end{multline*}
		where the last inequality follows from the fact that all $w_i\in \textbf{B}(0,|h|)$. Now, rewrite the last line as $\frac{C}{2}$ times two identical integrals. Using the fact that $H_x=-H_x$ we can substitute $h\mapsto -h$ in the second integral and get
		\begin{multline*}
		\Epq(f)\geq \frac{C}{2}\int_{ U}\int_{H_x}\int_{W^x_h}|(F(x+h)-F(x))\wedge(F(x+w_1)-F(x))\wedge\dots\wedge (F(x+w_n)-F(x))|^p\\
		\cdot |h|^{-(n+2)q}\ dw_1 \dots dw_n\ dh\ dx\\
		+\frac{C}{2}\int_{ U}\int_{H_x}\int_{W^x_h}|(F(x-h)-F(x))\wedge(F(x+w_1)-F(x))\wedge\dots\wedge (F(x+w_n)-F(x))|^p\\
		\cdot |h|^{-(n+2)q}\ dw_1 \dots dw_n\ dh\ dx.
		\end{multline*}
		We use the trivial estimate $|a|^p+|b|^p\geq 2^{1-p}|a+b|^p$ and $(n+1)$-linearity of exterior product to obtain
		\begin{multline}\label{eq:epq_est}
		\Epq(f)\geq
		C\int_{ U}\int_{H_x}\int_{W^x_h}|(F(x+h)-2F(x)+F(x-h))\wedge(F(x+w_1)-F(x))\\
		\wedge\dots\wedge (F(x+w_n)-F(x))|^p\cdot |h|^{-(n+2)q}\ dw_1 \dots dw_n\ dh\ dx.
		\end{multline}
		We need to estimate the term
		\begin{multline*}
		|(F(x+h)-2F(x)+F(x-h))\wedge(F(x+w_1)-F(x))\wedge\dots\wedge (F(x+w_n)-F(x))|\\
		= \left|\binom{0}{f(x+h)-2f(x)+f(x-h)}\wedge\binom{w_1}{f(x+w_1)-f(x)}\wedge\dots\wedge\binom{w_n}{f(x+w_n)-f(x)}\right|.
		\end{multline*}
		For brevity of notation let us set $\Delta^2_hf(x)=f(x+h)-2f(x)+f(x-h).$ 
		Applying Laplace expansion with respect to the first column yields		
		\begin{multline*}
		\left|\binom{0}{\Delta^2_hf(x)}\wedge\binom{w_1}{f(x+w_1)-f(x)}\wedge\dots\wedge\binom{w_n}{f(x+w_n)-f(x)}\right|\\
		= 	\left|\det \begin{pmatrix}
		0 & w_1 & \dots & w_n \\ 
		\Delta^2_hf(x) &  f(x+w_1)-f(x) & \dots & f(x+w_n)-f(x)
		\end{pmatrix} \right|\\
		=|\Delta^2_hf(x)||w_1\wedge\dots\wedge w_n|.
		\end{multline*}
		For $(w_1,\dots,w_n)\in W^x_h$ we have $|w_1\wedge\dots\wedge w_n|\geq \tilde{\alpha}|h|^n$, thus
		\begin{multline}\label{eq:wedges_est}
		|(F(x+h)-2F(x)+F(x-h))\wedge(F(x+w_1)-F(x))\wedge\dots\wedge (F(x+w_n)-F(x))|\\
		\geq \tilde{\alpha}|f(x+h)-2f(x)+f(x-h)||h|^n.
		\end{multline}
		Putting together (\ref{eq:epq_est}) and (\ref{eq:wedges_est}) we obtain
		\begin{multline*}
		\Epq(f)\geq C\int_{ U}\int_{H_x}\int_{W^x_h}\frac{|f(x+h)-2f(x)+f(x-h)|^p}{|h|^{(n+2)q-np}}\ dw_1 \dots dw_n\ dh\ dx\\
		= C\int_{ U}\int_{H_x} \H^{n^2}(W^x_h)\frac{|f(x+h)-2f(x)+f(x-h)|^p}{|h|^{(n+2)q-np}}\ dh\ dx.
		\end{multline*}	
		Lemma \ref{lem:alpha_tilde} assures that $\H^{n^2}(W^x_h) = \H^{n^2}(W^x_{|h|,\tilde{\alpha}})\ge C|h|^{n^2}$. Since $q= \frac{n(n+1)}{n+2} + \frac{p(n+1+s)}{n+2}$ we get
		\begin{equation*}
		\Epq(f)\geq C\int_{ U}\int_{H_x} \frac{|f(x+h)-2f(x)+f(x-h)|^p}{|h|^{n+sp+p}}\ dh\ dx = C [f]^p_{W^{1+sp,p}(U)}.
		\end{equation*}
	\end{proof}	
	\subsection*{Acknowledgments}
	The author was supported by NCN Grant no. 2013/10/M/ST1/00416 \emph{Geometric curvature energies for subsets of the Euclidean space.} He also acknowledges financial support from the Spanish Ministry of Economy and Competitiveness, through the María de Maeztu Programme for Units of Excellence in R\&D (MDM-2014-0445).
	
	The contents of this article constituted the author's Master's thesis written at the University of Warsaw under supervision of Paweł Strzelecki. The author would like to express his gratitude to Professor Strzelecki for all his help and advice. He would also like to thank Sławomir Kolasiński for reading this paper and for his valuable suggestions. Finally, many thanks are due to the anonymous referee who pointed out the connection to \cite{dorronsoro1985mean}, which allowed to simplify some of the proofs and extend the results.

	\bibliography{bib_arxiv (copy)} 
	\bibliographystyle{myalpha}
	
\end{document}